\documentclass[12pt, a4paper]{amsart}
\usepackage{amsmath}
\usepackage{geometry,amsthm,graphics,tabularx,amssymb,
shapepar}
\usepackage{amscd}
\usepackage{amsxtra}
\usepackage[usenames]{color}
\usepackage{mathdots}

\usepackage[all,2cell,dvips]{xy}

\newcommand{\CF}{{\mathcal {F}}}

\newcommand{\CS}{{\mathcal {S}}}

\newcommand{\RC}{{\mathrm {C}}}

\newcommand{\Ad}{{\mathrm{Ad}}}

\newcommand{\GL}{{\mathrm{GL}}}

\newcommand{\Hom}{{\mathrm{Hom}}}

\newcommand{\uInd}{{{}^\mathrm{u}\mathrm{Ind}}}
\newcommand{\Ind}{{\mathrm{Ind}}}

\newcommand{\N}{{\mathrm{N}}}

\newcommand{\con}{\textit{C}}

\newcommand{\od}{\operatorname{d}}

\newcommand{\oP}{\operatorname{P}}

\newcommand{\C}{\mathbb{C}}
\newcommand{\R}{\mathbb R}

\renewcommand{\N}{\mathbb{N}}

\newcommand{\abs}[1]{\lvert#1\rvert}

\newcommand{\be}{\begin {equation}}
\newcommand{\ee}{\end {equation}}
\newcommand{\bp}{\begin {proof}}
\newcommand{\ep}{\end {proof}}
\newcommand{\bee}{\begin {equation*}}
\newcommand{\eee}{\end {equation*}}

\theoremstyle{Theorem}

\newtheorem{thm}{Theorem}[section]

\theoremstyle{Theorem}
\newtheorem{lem}{Lemma}[section]

\theoremstyle{Theorem}

\theoremstyle{Plain}

\theoremstyle{Theorem}

\theoremstyle{Definition}

\begin{document}

\title[Homogeneous distributions]{Homogeneous distributions on finite dimensional vector spaces}

\author [H. Xue] {Huajian Xue}

\address{Academy of Mathematics and Systems Science\\
Chinese Academy of Sciences\\
Beijing, 100190,  China} \email{xuehuajian@126.com}



\begin{abstract}
Let $V$ be a finite dimensional vector space over a local field $F$. Let $\chi: F^\times \rightarrow \C^\times$ be an arbitrary  character of $F^\times$. We determine the structure of the natural representation of $\GL(V)$ on the space  $\mathcal{S}^*(V)^\chi$ of $\chi$-invariant distributions on $V$.
\end{abstract}

\maketitle

\section{Introduction}\label{secintro}

Let $V$ be a  vector space over a local field $F$, of finite dimension $n\geq 1$. For every $g\in \GL(V)$ and every distribution $\eta$ on $V$, define
\be\label{action}
  g. \eta:=\textrm{the push-forward of $\eta$ through the map $g: V\rightarrow V$. }
\ee
For each character $\chi: F^\times \rightarrow \C^\times$, a distribution $\eta$ on $V$ is said to be $\chi$-invariant if
\[
  a. \eta=\chi(a)\, \eta,\quad\textrm{for all $a\in F^\times$}.
\]
Here and as usual, $F^\times$ is  identified with  the center of $\GL(V)$.  By \cite[Theorem 4.0.2]{AGS}, we know that every $\chi$-invariant distribution on $V$ is tempered when $F$ is archimedean. By convention, every distribution on $V$ is defined to be tempered when $F$ is non-archimedean. The goal of this paper is to understand the space
\be\label{invdis}
  \mathcal{S}^*(V)^\chi:=\{\eta\in \CS^*(V)\mid a. \eta=\chi(a)\, \eta\}
\ee
of $\chi$-invariant tempered distributions on $V$, as a representation of $\GL(V)$ under the action \eqref{action}. Here and as usual, $\mathcal{S}(V)$ denotes the space of Schwartz or Schwartz-Bruhat functions on $V$, when  $F$ is respectively archimedean or non-archimedean; and $\CS^*(V)$ denotes the space of all (continuous in the archimedean case) linear functionals on  $\mathcal{S}(V)$. It is a fundamental fact in Tate's thesis that the space \eqref{invdis} is one dimensional when $n=1$ (see \cite[Section 1]{Wei}). Thus we will focus on the case when $n\geq 2$.

Dualizing the action \eqref{action}, we have a representation of $\GL(V)$ on  $\mathcal{S}(V)$ by
\[
  (g. f)(x):=f(g^{-1}x),\quad \textrm{for all $g\in \GL(V)$, $f\in \mathcal{S}(V)$, $x\in V$}.
\]
Define the normalized $(F^\times,\chi)$-coinvariant space
\be\label{coinvs}
   \mathcal{S}_\chi(V):=\left(\mathcal{S}(V)\otimes \abs{\det}_F^{-\frac{1}{2}}\right)_{F^\times, \chi},
\ee
where $\abs{\det}_F$ denotes the positive character
\[
   \GL(V)\rightarrow \C^\times, \quad g\mapsto \abs{\det(g)}_F,
\]
$\abs{\,\cdot\, }_F$ denotes the normalized absolute value on $F$, and for a smooth representation $U$ of $\GL(V)$,  $U_{F^\times,\chi}$ denotes the maximal  (Hausdorff in the archimedean case) quotient of $U$ on which $F^\times$ acts through the character $\chi$. Here and as usual, we do not distinguish a one dimensional representation with its corresponding character.
Then we have
\be\label{disvs}
  \mathcal{S}^*(V)^\chi\cong \left(\mathcal{S}_{\chi^{-1}\,\cdot\, \abs{\,\cdot\,}_F^{-\frac{n}{2}}}(V)\right)^*\otimes \abs{\det}_F^{-\frac{1}{2}},
\ee
as representations of $\GL(V)$. Here and henceforth, we use a superscript $\,^*$ to indicate the dual space in various contexts.  Thus we only need to study the representation \eqref{coinvs}.

As before, write $\mathcal{S}(V\setminus \{0\})$ for the space of Schwartz
 or Schwartz-Bruhat functions on $V\setminus \{0\}$, when  $F$ is respectively archimedean or non-archimedean (see \cite{AG} for the definition of Schwartz functions in the archimedean case). Similar to \eqref{coinvs}, define
 \be\label{coinvs2}
   \mathcal{S}_\chi(V\setminus\{0\}):=\left(\mathcal{S}(V\setminus\{0\})\otimes \abs{\det}_F^{-\frac{1}{2}}\right)_{F^\times,\chi}.
\ee
Then the embedding $ \mathcal{S}(V\setminus\{0\})\hookrightarrow \mathcal{S}(V)$ induces a homomorphism
\be \label{map-j}
  \mathrm j_\chi:  \mathcal{S}_\chi(V\setminus\{0\})\rightarrow  \mathcal{S}_\chi(V)
\ee
of representations of $\GL(V)$.

It is easy to see that the representation $\mathcal{S}_\chi(V\setminus\{0\})$ is isomorphic to a degenerate principal series. More precisely, fix an arbitrary nonzero vector $v_0\in V$, and write $\oP(v_0)=F^\times \times \oP^\circ(v_0)$ for the maximal parabolic subgroup of $\GL(V)$ stabilizing $Fv_0$, where $\oP^\circ(v_0)$ denotes the stabilizer of $v_0$ in $\GL(V)$. Then  the linear map
\[
  \begin{array}{rcl}
   \mathcal{S}(V\setminus\{0\})\otimes \abs{\det}_F^{-\frac{1}{2}} &\rightarrow &\con^\infty(\GL(V)),\\
      \phi\otimes 1 &\mapsto & \left(g\mapsto \int_{F^\times}  \phi((ga)^{-1}v_0)\cdot \abs{\det(ga)}_F^{-\frac{1}{2}} \cdot\chi^{-1}(a) \,\od^\times\! a\right)
  \end{array}
\]
induces a $\GL(V)$-intertwining isomorphism
\be\label{isod}
  \mathcal{S}_\chi(V\setminus\{0\})\cong \Ind_{\oP(v_0)}^{\GL(V)} \chi\otimes 1.
\ee
Here $\od^\times\! a$ is a Haar measure on the multiplicative group $F^\times$. And ``$\Ind$'' indicates the normalized smooth induction, on which $\GL(V)$ acts by right translation.  While ``$1$'' stands for the trivial character of $\oP^\circ(v_0)$. The structure of this degenerate principal series is well-known (see Lemma \ref{degen}).

Denote  by $\RC_F$ the submonoid of  $\Hom(F^\times,\C^\times)$ generated by characters of the form $\iota|_{F^\times}:F^\times \rightarrow \C^\times$, where $\iota: F\hookrightarrow \C$ is a continuous field embedding. Explicitly, denote by $\N$ the set of non-negative integers. If $F=\R$, let $\iota$ be the natural imbedding of $\R$ into $\C$. If $F=\C$, let $\iota_1$, $\iota_2$ be the identity map and complex conjugate respectively. Then
\[
\RC_F=\begin{cases}
\{(\iota\vert_{\R^\times})^r\vert r\in \N\},  & \mbox{ if } F=\R; \\
\{(\iota_1\vert_{\C^\times})^r(\iota_2\vert_{\C^\times})^s\vert r,s\in \N\},  & \mbox{ if } F=\C; \\
\{1\}, & \mbox{ if } F \mbox{ is non-archimedean}.
\end{cases}
\]
In Section \ref{secd0}, we will  define an irreducible finite dimensional representation $\sigma_{V,\chi}$ of $\GL(V)$ for  each $\chi\in \RC_F$.
Note that
\[
 \RC^+(n):=\abs{\, \cdot\,}_F^{\frac{n}{2}}\, \RC_F\quad\textrm{and}\quad \RC^-(n):=\{\chi^{-1}\mid \chi\in \RC^+(n)\}
\]
are disjoint subsets of $\Hom(F^\times, \C^\times)$.

Now the main result of this paper is formulated as follows.

\begin{thm}\label{thma00} Let $V$ be an $n$-dimensional ($n\geq 2$) vector space over a local field $F$, and let $\chi$ be a character of $F^\times$. Define the homomorphism $\mathrm j_\chi:  \mathcal{S}_\chi(V\setminus\{0\})\rightarrow  \mathcal{S}_\chi(V)$ as \eqref{map-j}. Then we have the following:

(a) If $\chi\notin \RC^+(n)\cup \RC^-(n)$, then $\mathrm j_\chi$ is an isomorphism of irreducible representations.

(b) If $\chi\in \RC^+(n)$,  then $\mathrm j_\chi$ is an isomorphism and $\mathcal{S}_\chi(V)$ has a unique irreducible subrepresentation, and the corresponding  quotient representation is isomorphic to $\sigma_{V, \chi\cdot \abs{\,\cdot\,}_F^{-\frac{n}{2}}}\otimes \abs{\det}_F^\frac{1}{2}$.

(c) If $\chi\in \RC^-(n)$, then both $\mathcal{S}_\chi(V)$ and $\mathcal{S}_\chi(V\setminus\{0\})$ have length $2$ and have a unique irreducible subrepresentation. The unique irreducible subrepresentation of $\mathcal{S}_\chi(V)$ is isomorphic to
\[
  \mathrm{ker} (\mathrm j_\chi)\cong \mathrm{coker}(\mathrm  j_\chi)\cong \left(\sigma_{V, \chi^{-1}\cdot \abs{\,\cdot\,}_F^{-\frac{n}{2}}}\otimes \abs{\det}_F^\frac{1}{2}\right)^*.
  \]
  \end{thm}

Note that $\mathcal{S}(V)$ carries an action of $\GL_1(F)\times \GL(V)$ which is defined by
\[
((g_1, g_2). f)(x):=f(g_2^{-1}xg_1),\quad (g_1, g_2)\in \GL_1(F)\times \GL(V), f\in \mathcal{S}(V), x\in V.
\]
Denote by $\Theta(\chi)$ the full theta lift of the representation $\chi$ of $\GL_1(F)$ to $\GL(V)$. Then we have
\[
\left(\mathcal{S}(V)\otimes (\abs\cdot_F ^{\frac{n}{2}}\otimes\abs{\det}_F^{-\frac{1}{2}})\right)_{\GL_1(F), \chi}\cong\chi \otimes \Theta(\chi)
\]
as representations of $\GL_1(F)\times \GL(V)$. It follows that $\Theta(\chi)\cong \mathcal{S}_{\chi^{-1}}(V)$ in our setting.
It is a fundamental fact that the full theta lift always has a unique irreducible quotient whenever it is nonzero (see \cite{How, Wald, Min, GT, GaS}). Howe also expects that in many cases, the full theta lift also has a unique irreducible subrepresentation, and the irreducible  subrepresentation is ``large" and the irreducible quotient representation is ``small". Theorem \ref{thma00} provides some evidences for these expectations.

\textbf{Acknowledgements}: The author would like to thank Binyong Sun for the encouragements and fruitful discussions. He also thanks the referee for his careful reading and very helpful comments and suggestions on the previous version of this paper.

\section{Distributions supported at $0$}\label{secd0}

We continue with the notation of the Introduction.
For every $\chi\in \RC_F$, we shall define a representation $\sigma_{V,\chi}$ of $\GL(V)$ in what follows.
If $F$ is non-archimedean, we define $\sigma_{V,1}$ to be the one dimensional trivial representation for the unique member $1\in \RC_F$. If $F\cong\R$, let $\chi=(\iota|_{F^\times})^r \in\RC_F,(r\in\N)$ (as is defined in the Introduction), we define
\[
   \sigma_{V,\chi}:=\mathrm{Sym}^r(V\otimes_{F,\iota} \C).
\]
Here and henceforth, $\mathrm{Sym}^r$ indicates the $r$-th symmetric power. If $F\cong\C$, let $\chi=(\iota_1|_{F^\times})^r\cdot (\iota_2|_{F^\times})^s \in\RC_F,(r,s\in\N)$,  we define
\[
   \sigma_{V,\chi}:=\mathrm{Sym}^r(V\otimes_{F,\iota_1} \C)\otimes \mathrm{Sym}^s(V\otimes_{F,\iota_2} \C).
\]
In all cases, $\sigma_{V,\chi}$ is an irreducible finite dimensional representation of $\GL(V)$ of central character $\chi$.

Let $\CS^*(V,\{0\})$ denotes the space of tempered distributions on $V$ whose support is contained in $\{0\}$. For each character $\chi: F^\times \rightarrow \C^\times$, put
\[
  \CS^*(V,\{0\})^\chi:=\CS^*(V,\{0\})\cap \CS^*(V)^\chi,
\]
which is a representation of $\GL(V)$.

Denote by $\delta_0$ the Dirac distribution. Set $V=F^n$. If $F=\R$, for $I=(k_1,k_2,\cdots,k_n)\in \N^n$, let $\partial^I$ be the differential operator which takes $k_i$-th derivative for the $i$-th variable for all $1\leq i \leq n$. If $F=\C$, define $\partial^{I_1}\bar\partial^{I_2}$ for $I_1, I_2\in \N^n$ similarly. By \cite[Theorem 1.7]{SS}, we have
\[
\CS^*(V,\{0\})=\begin{cases}
\mbox{span}\{\partial^I\delta_0\vert I\in\N^n\}, & \mbox{ if } F=\R; \\
\mbox{span}\{\partial^{I_1}\bar\partial^{I_2}\delta_0\vert I_1, I_2\in \N^n\}, & \mbox{ if } F=\C;\\
\mbox{span}\{\delta_0 \}, & \mbox{ if } F \mbox{ is non-archimedean}.
\end{cases}
\]
Then it is easy to deduce the following
\begin{lem}\label{distribution0}
If $\chi\in \RC_F$, then $\CS^*(V,\{0\})^\chi\cong \sigma_{V,\chi}$; otherwise it is zero. Moreover,
\[
  \CS^*(V,\{0\})=\bigoplus_{\chi\in \RC_F} \CS^*(V,\{0\})^\chi.
\]
\end{lem}

\section{Fourier transform}\label{secpre}

 Denote by $V^*$ the dual space of $V$. The Fourier transform yields a linear isomorphism
\[
  \CF: \CS(V)\otimes \abs{\det}_F^{-\frac{1}{2}}\rightarrow \CS(V^*)\otimes \abs{\det}_F^{-\frac{1}{2}}, \quad \phi\otimes 1\mapsto \widehat \phi\otimes 1,
\]
where
\[
   \widehat \phi(\lambda):=\int_V \phi(x)\psi(\lambda(x)) \od\! x.
\]
Here $\od\! x$ is a fixed Haar measure on $V$, and $\psi$ is a fixed non-trivial unitary character $\psi$ on $F$. It  is routine to check that
\[
  \CF(g.\eta)=g^{-t}.(\CF(\eta)),\quad \textrm{for all } g\in \GL(V), \, \eta\in \CS(V)\otimes \abs{\det}_F^{-\frac{1}{2}}.
\]
Here $g^{-t}\in \GL(V^*)$ denotes the inverse transpose of $g$. Consequently, $\CF$ induces a linear isomorphism
\be\label{isovv}
  \mathcal{S}_\chi(V)\cong \mathcal{S}_{\chi^{-1}}(V^*)
\ee
which is intertwining with respect to the isomorphism $\GL(V)\rightarrow \GL(V^*), g\mapsto g^{-t}$.

\begin{lem}\label{infdim0}
If $n\geq 2$, then the space $\mathcal{S}_\chi(V)$ is infinite dimensional.
\end{lem}
 \begin{proof}
Since $\RC^+(n)$ and $\RC^-(n)$ are disjoint, by using the isomorphism \eqref{isovv}, it suffices to prove the lemma in the case when $\chi\notin \RC^-(n)$.

As a reformulation of  \eqref{disvs}, we have
\be
  \mathcal{S}^*(V)^{\chi^{-1}\,\cdot\,\abs{\,\cdot\,}_F^{-\frac{n}{2}}}\cong \left(\mathcal{S}_{\chi}(V)\right)^*\otimes \abs{\det}_F^{-\frac{1}{2}},
\ee
Note that $\chi\notin \RC^-(n)$ if and only if $\chi^{-1}\,\cdot\,\abs{\,\cdot\,}_F^{-\frac{n}{2}}\notin \Hom_{\mathrm{alg}}(F^\times,\C^\times)$. Thus we only need to show that $\mathcal{S}^*(V)^\chi$ is infinite dimensional when $\chi\notin \Hom_{\mathrm{alg}}(F^\times,\C^\times)$.

For each one dimensional subspace $L$ of $V$, let $\eta_L$ denote a nonzero distribution in the one dimensional space  $\mathcal{S}^*(L)^\chi$. Assuming $\chi\notin \Hom_{\mathrm{alg}}(F^\times,\C^\times)$, we know from  Lemma \ref{distribution0} that the support of $\eta_L$ equals $L$. Then the infinite family
\[
  \{\textrm{the push-forward of $\eta_L$ through the embedding $L\hookrightarrow V$ }\},
\]
where $L$ runs over all one dimensional subspace of $V$, is linearly independent in $\mathcal{S}^*(V)^\chi$. This shows that the space $\mathcal{S}^*(V)^\chi$ is infinite dimensional.

\end{proof}

\section{Proof of theorem \ref{thma00}}\label{secreal}

In this section, assume that $n\geq 2$. Recall the infinite dimensional representation $\CS_\chi(V\setminus \{0\})\cong\Ind_{\oP(v_0)}^{\GL(V)} \chi\otimes 1$ from the Introduction. Combining \cite[Section 2.4 and 3.4]{HL} and \cite[Theorem 1.1]{GoS}, we get the following lemma.

\begin{lem} \label{degen}

(a) If $\chi\notin \RC^+(n)\cup \RC^-(n)$, then the representation $\Ind_{\oP(v_0)}^{\GL(V)} \chi\otimes 1$ is irreducible; otherwise, it has length $2$ and has a unique irreducible subrepresentation.

(b) If $\chi\in \RC^+(n)$, then the irreducible quotient representation of $\Ind_{\oP(v_0)}^{\GL(V)} \chi\otimes 1$ is isomorphic to $\sigma_{V, \chi\cdot \abs{\,\cdot\,}_F^{-\frac{n}{2}}}\otimes \abs{\det}_F^\frac{1}{2}$.

(c) If $\chi\in \RC^-(n)$, then the irreducible subrepresentation of $\Ind_{\oP(v_0)}^{\GL(V)} \chi\otimes 1$ is isomorphic to $\left(\sigma_{V, \chi^{-1}\cdot \abs{\,\cdot\,}_F^{-\frac{n}{2}}}\otimes \abs{\det}_F^\frac{1}{2}\right)^*$.
\end{lem}

\bp
Let $F=\R$. Set $G=\GL(V)$ and $P=\oP(v_0)$. We may assume that $v_0=(1,0,\cdots,0)^t$, thus $P$ has the form
$P=\begin{pmatrix}
\R^\times & * \\
0 & \GL_{n-1}(\R)
\end{pmatrix}$, and $P^\circ=\oP^\circ(v_0)$ is the subgroup of $P$ with the first column being $(1,0,\cdots,0)^t$.
The modular character $\Delta_P$ of $P$ satisfies
\[
\Delta_P(p)=\abs {\det(\Ad_p)}=\abs a ^{n-1} \abs {\det g} ^{-1}=\abs a ^n \abs {\det p} ^{-1},
\]
where $p=\begin{pmatrix}
a & x \\
0 & g
\end{pmatrix}\in P$. Denote by $\uInd_P^G(\chi\otimes 1)$ the non-normalized induction and observe that $$\uInd_P^G((\chi\otimes 1)\otimes \Delta_P^{\frac{1}{2}})=\Ind_P^G(\chi\otimes 1).$$
For any representation $\pi$ of $P$ and any character $\chi^\prime$ of $G$, there is a natural isomorphism
\be
\uInd_P^G(\pi\otimes \chi^\prime\vert_P)\cong \chi^\prime \otimes \uInd_P^G\pi.
\ee
It follows that
\be \label{ind-isom1}
\Ind_P^G(\chi\otimes 1)\cong \abs\det ^{-\frac{1}{2}}\otimes \uInd_P^G(\chi\cdot\abs\cdot^{\frac{n}{2}}\otimes 1)
\ee
by setting $\pi=\chi \otimes 1$ and $\chi^\prime=\abs\det ^{-\frac{1}{2}}$.

Let $Q=\begin{pmatrix}
\GL_{n-1}(\R) & * \\
0 & \R^\times
\end{pmatrix}$ be another parabolic subgroup of $G$, and let $Q^{\circ}$ be the subgroup of $Q$ with the last row being $(0,\cdots,0,1)$. Then $Q=Q^{\circ}\times \R^\times$. Set $w=\begin{pmatrix}
 & & 1 \\
 & \iddots & \\
1 & &
\end{pmatrix}$. We have an isomorphism
\be \label{ind-isom2}
\begin{array}{cccc}
   \tau: & \uInd_P^G(\chi\otimes 1) &\xrightarrow{\sim} &\uInd_{Q}^G(1\otimes \chi^{-1}),\\
       &f &\mapsto & \left(g\mapsto  f(wg^{-t}w)\right),
  \end{array}
\ee
which is intertwining with respect to the isomorphism $G \to G, g\mapsto wg^{-t}w$.

By (\cite{HL}, Theorem 3.41, 3.42, 3.43 and 3.44, applied with $k=1$), the representation
$\uInd_{Q}^G(1\otimes \chi)$ is irreducible if $\chi\notin\abs \cdot^{-\frac{n}{2}}\RC^+(n)\cup \abs \cdot^{-\frac{n}{2}}\RC^-(n)$; otherwise, it has length $2$ and has a unique irreducible subrepresentation.
Hence the assertion (a) follows from the isomorphisms \eqref{ind-isom1} and \eqref{ind-isom2}.

If $\chi=(\iota \vert _{\R^\times})^r(r\in \N)$. For $I=(k_1, k_2,\cdots, k_n)\in \N^n$ which satisfies $\abs I:=k_1+k_2+\cdots+k_n=r$, we define
\[
f_I\begin{pmatrix}
* & * &\cdots & * \\
\vdots &\vdots & &\vdots \\
* & * &\cdots & * \\
a_1 & a_2 & \cdots & a_n
\end{pmatrix}:=a_1^{k_1}a_2^{k_2}\cdots a_n^{k_n}.
\]
It is easy to check that $f_I\in \uInd_{Q}^G(1\otimes \chi)$, and $W:=\mbox{span}\{ f_I\vert I\in\N^n, \abs I=r \}$ is a (finite dimensional) subrepresentation of $\uInd_{Q}^G(1\otimes \chi)$ which is isomorphic to $\sigma_{V,\chi}$. Applying the isomorphisms \eqref{ind-isom1} and \eqref{ind-isom2}, we get the Part (c).

Note that the representations $\Ind_P^G(\chi\otimes 1)$ and $\Ind_P^G(\chi^{-1}\otimes 1)$ are contragredient (\cite[Lemma 5.2.4]{Wal}). Thus (c) implies (b), and we complete the proof for $F=\R$.

The proofs for the cases that $F=\C$ and $F$ is non-archimedean are analogous, and we omit the details here.

\ep

The short exact sequence
\[
  0\rightarrow \CS(V\setminus \{0\})\rightarrow \CS(V)\rightarrow (\CS^*(V,\{0\}))^*\rightarrow 0
\]
induces an exact sequence
\be\label{rexact}
\CS_\chi(V\setminus \{0\})\xrightarrow{\mathrm{j}_\chi} \CS_\chi(V)\rightarrow \left(\CS^*(V,\{0\})^{\chi^{-1}\,\cdot\,\abs{\,\cdot\,}_F^{-\frac{n}{2}}}\right)^*\otimes \abs{\det}_F^{-\frac{1}{2}}\rightarrow 0.
\ee
Here we have used the natural isomorphism
\[
 \left((\CS^*(V,\{0\}))^*\otimes \abs{\det}_F^{-\frac{1}{2}}\right)_{F^\times,\chi}\cong \left(\CS^*(V,\{0\})^{\chi^{-1}\,\cdot\,\abs{\,\cdot\,}_F^{-\frac{n}{2}}}\right)^*\otimes \abs{\det}_F^{-\frac{1}{2}}
\]
of representations of $\GL(V)$.

If $\chi\notin  \RC^-(n)$, then Lemma \ref{distribution0} and the exactness of \eqref{rexact}  imply that $\mathrm j_\chi$ is surjective.
We argue according to the three cases of Theorem \ref{thma00}.

\noindent \textbf{Case a}: $\chi\notin \RC^+(n)\cup \RC^-(n)$.

In this case, $\CS_\chi(V\setminus \{0\})$ is irreducible by Part (a) of  Lemma \ref{degen}. By Lemma \ref{infdim0}, $\CS_\chi(V\setminus \{0\})$ is nonzero. Therefore  the surjective homomorphism $\mathrm j_\chi$ is an isomorphism. This proves Part (a) of Theorem \ref{thma00}.

\noindent \textbf{Case b}: $\chi\in \RC^+(n)$.

In this case, by Lemma \ref{infdim0}, $\mathrm{ker}(\mathrm j_\chi)$ is a subrepresentation of $\CS_\chi(V\setminus \{0\})$ of infinite codimension. Then Parts (a) and (b) of  Lemma \ref{degen} implies that $\mathrm{ker}(\mathrm j_\chi)=\{0\}$, which further implies Part (b) of Theorem \ref{thma00}.

\noindent \textbf{Case c}: $\chi\in \RC^-(n)$.

In this case, applying Part (b) of Theorem \ref{thma00} to $V^*$, and using the isomorphism \eqref{isovv}, we know that $\CS_\chi(V)$ has length $2$ and has a unique irreducible subrepresentation.  The exact sequence \eqref{rexact} and  Lemma \ref{distribution0} imply that
\[
 \mathrm{coker}(\mathrm j_\chi)\cong \left(\sigma_{V, \chi^{-1}\cdot \abs{\,\cdot\,}_F^{-\frac{n}{2}}}\otimes \abs{\det}_F^\frac{1}{2}\right)^*.
 \]
 Thus the image of $\mathrm j_\chi$ is irreducible, and hence $\mathrm{ker}(\mathrm j_\chi)$ is irreducible as $\CS_\chi(V\setminus \{0\})$ has length $2$. Applying Part (c) of Lemma \ref{degen}, this proves Part (c) of Theorem \ref{thma00}.


\begin{thebibliography}{}


\bibitem[AG]{AG}
A. Aizenbud, D. Gourevitch, \textit{Schwartz functions on Nash Manifolds}, International Mathematics Research Notices,
Vol. 2008, n.5, Article ID rnm155, 37 pages. DOI: 10.1093/imrn/rnm155.

\bibitem[AGS]{AGS}
A. Aizenbud, D. Gourevitch, and E. Sayag, \textit{ Generalized Harish-Chandra descent, Gelfand pairs, and an Archimedean analog of Jacquet-Rallis's theorem}, Duke Mathematical Journal, 149(3), 509-567, 2009.

\bibitem[FH]{FH}
W. Fulton, J. Harris, \textit{Representation theory}, Vol. 129, Springer Science and Business Media, 1991.

\bibitem[GaS]{GaS}
W. T. Gan, B. Sun, \textit{The Howe duality conjecture: quaternionic case}, Preprint, 2015.

\bibitem[GoS]{GoS}
D. Gourevitch, S. Sahi, \textit{Composition series for degenerate principal series of $GL(n)$}, Preprint, 2013.


\bibitem[GT]{GT}
W. T. Gan, S. Takeda, \textit{A proof of the Howe duality conjecture}, Preprint, 2014.

\bibitem[HL]{HL}
R. Howe, S. T. Lee, \textit{Degenerate principal series representations of $Gl_n(\C)$ and $Gl_n(\R)$}, Journal of Functional Analysis 166, 244-309, 1999.


\bibitem[How]{How}
R. Howe, \textit{$\theta$-series and invariant theory}, Proc. Symp. Pure Math. 33 (1979), 275-285.

\bibitem[Kud]{Kud}
S. S. Kudla, \textit{Tate¡¯s thesis}, An Introduction to the Langlands Program, Springer, Boston, 2004.

\bibitem[Lig]{Lig}
M. J. Lighthill, \textit{Introduction to fourier analysis and generalised functions}, Cambridge University Press, Cambridge, UK, 1958.

\bibitem[Min]{Min}
A. Minguez, \textit{Correspondance de Howe explicite: paires duales de type II}, Annales scientifiques de l¡¯{\'E}. NS, Vol. 41, No. 5 (2008), 717-741.

\bibitem[MVW]{MVW}
C. Moeglin and M. F. Vigneras and J. L. Waldspurge, \textit{Correspondances de How sur un corps p-adique}, Lecture Notes in Mathematics, Vol. 1291, Springer Berlin, 1987.


\bibitem[SS]{SS}
E. M. Stein, R. Shakarchi, \textit{Functional analysis}, Princeton Lectures in Analysis, Vol. 4, Princeton University Press, Princeton, NJ, 2011.

\bibitem[Tat]{Tat}
J. Tate, \textit{Fourier analysis in number fields and Hecke's zeta-functions}, Diss, Princeton University, 1950.

\bibitem[Wal]{Wal}
N. Wallach, \textit{Real reductive groups I}, Pure and Applied Math. 132, Academic Press, Boston, MA, 1988.

\bibitem[Wald]{Wald}
J. -L. Waldspurger, \textit{D{\'e}monstration d¡¯une conjecture de dualit{\'e} de Howe dans le cas p-adique, p= 2}, Israel Mathematical Conference Proceedings, Vol. 2 (1990), 267--324.

\bibitem[Wei]{Wei}
A. Weil, \textit{Fonction z{\^e}ta et distributions}, S{\'e}minaire Bourbaki, Vol. 9 (1966), 523--531.
\end{thebibliography}
\end{document}